\documentclass[12pt]{article}
\usepackage{authblk}
\usepackage{amsmath,amssymb,amsthm,multirow,graphicx,ifpdf,mathtools}
\usepackage[mathscr]{eucal}
\usepackage{calrsfs}
\newtheorem{theorem}{Theorem}[section]
\newtheorem{proposition}[theorem]{Proposition}
\newtheorem{lemma}[theorem]{Lemma}
\newtheorem{corollary}[theorem]{Corollary}

\theoremstyle{definition}
\newtheorem{definition}[theorem]{Definition}
\theoremstyle{remark}
\newtheorem{remark}[theorem]{Remark}
\newtheorem{example}[theorem]{Example}

\numberwithin{equation}{section}

\usepackage{enumitem}

\usepackage{authblk}

\allowdisplaybreaks

\title{A Hilbert space approach\\ to difference equations}

\author[1]{Konrad Kitzing}
\author[1,2]{Rainer Picard}
\author[1,3]{Stefan Siegmund}
\author[1]{\\ Sascha Trostorff}
\author[2]{Marcus Waurick}
\affil[1]{Technische Universit\"{a}t Dresden, Faculty of Mathematics,
Institute of Analysis, Zellescher Weg 12-14, 01069 Dresden, Germany}
\affil[2]{Department of Mathematics and Statistics, University of Strathclyde, 26 Richmond Street, Glasgow G1 1XH, Scotland}
\affil[3]{Technische Universit\"{a}t Dresden, Faculty of Mathematics,
Center for Dynamics, Zellescher Weg 12-14, 01069 Dresden, Germany}
\renewcommand{\phi}{\varphi}
\renewcommand{\rho}{\varrho}
\newcommand{\eps}{\varepsilon}
\newcommand{\R}{\mathbb{R}}
\newcommand{\C}{\mathbb{C}}

\newcommand{\N}{\mathbb{N}}

\newcommand{\Z}{\mathbb{Z}}

\newcommand{\cZ}{\mathcal{Z}}
\newcommand{\cc}{\operatorname{cc}}

\newcommand{\setm}[2]{\{\hspace*{0,07em}#1;#2\hspace*{0,07em}\}}
\DeclarePairedDelimiter{\set}{\{}{\}}

\begin{document}

\maketitle

\abstract{We consider general difference equations $u_{n+1} = F(u)_n$ for $n \in \Z$ on exponentially weighted $\ell_2$ spaces of two-sided Hilbert space valued sequences $u$ and discuss initial value problems. As an application of the Hilbert space approach, we characterize exponential stability of linear equations and prove a stable manifold theorem for causal nonlinear difference equations.}

\section{Introduction}

Let $H$ be a separable Hilbert space over $\C$ with scalar product $\langle {}\cdot{}, {}\cdot{} \rangle_{\!H}$, linear in the second and conjugate linear in the first argument, let $F \colon H^\Z \to H^\Z$. In this paper, we discuss difference equations
\begin{equation}\label{diffeq}
   u_{n+1} = F(u)_n
   \qquad
   (n \in \Z)
\end{equation}
and their solutions $u \in \ell$ on exponentially weighted sequence spaces $\ell \subseteq H^\Z$ (see Lemma \ref{lem:ell2rho}).
With the shift $\tau \colon H^\Z \to H^\Z$, $(u_n)_{n \in \Z} \mapsto (u_{n+1})_{n \in \Z}$, (see Lemma \ref{lem:Shift}), equation
\eqref{diffeq} on $\ell$ can be reformulated equivalently as
\begin{equation}\label{tau-diffeq}
   \tau u = F(u)
   \qquad
   (u \in \ell).
\end{equation}
Since $\tau$ is invertible with inverse $\tau^{-1} \colon H^\Z \to H^\Z$, $(u_n)_{n \in \Z} \mapsto (u_{n-1})_{n \in \Z}$, each solution $u \in \ell$ of \eqref{tau-diffeq} satisfies $u = \tau^{-1} F(u)$.  If $\ell \subseteq H^\Z$ is a Banach space and $\tau^{-1} \circ F \colon \ell \to \ell$ a contraction, Banach's fixed point theorem yields a unique solution to \eqref{tau-diffeq} (see Theorem \ref{t:existence1}).

For example, the search for bounded solutions of an implicit (backward) Euler scheme approximation $u_{n+1} = u_n + h f(u_{n+1})$, $n \in \Z$, $h > 0$, of a differential equation with right-hand side $f \colon \R^d \to \R^d$, is of the form \eqref{tau-diffeq} with $F(u)_n = u_n + h f(u_{n+1})$ for $n \in \Z$ on the sequence space $\ell = \ell_{\infty}(\Z; \R^d) = \setm{x \colon \Z \to \R^d}{\sup_{k \in \Z} |x_k|_{\R^d} < \infty}$. On the other hand, if the implicit Euler scheme $u_{n+1} =  F(u)_n$ is considered on $\N \coloneqq \Z_{\geq 0}$ with an initial condition $u_0 = x \in \R^n$, it is well-known that for each $u_n$ the implicit equation $u_{n+1} = u_n + h f(u_{n+1})$ needs to be solved iteratively, e.g.\ with the Newton method, to compute $u_{n+1}$ from $u_n$. We will prove in Theorem \ref{t:initial_value_problem} that initial value problems can be equivalently reformulated into \eqref{diffeq} on $\Z$ by adding the initial value $x$ as an impulse $\delta_k x \coloneqq (\delta_{k,n} x)_{n \in \Z} \in H^\Z$ at $k = -1$, where $\delta_{k,n}$ denotes the Kronecker delta symbol for $k, n \in \Z$, if $F(u)_n$ only depends on $u_k$ for $k \in \Z_{\leq n}$, i.e.\ for so-called causal $F$ (Definition \ref{d:causal}). Note that for the implicit Euler scheme, $F(u)_n$ depends on $u_{n+1}$ and $F$ is not causal. 

Let $L(H)$ denote the space of bounded linear operators from $H$ to $H$. In Section 3 we characterize exponential stability of linear difference equations $u_{n+1} = A u_n$ for $A \in L(H)$ (Theorem \ref{t:expstab}), by utilizing the $\cZ$ transform (Theorem \ref{t:ztransform}). In Section 4 we prove a stable manifold theorem (Theorem \ref{t:stable-manifold}) for causal (Definitions \ref{d:linearcausal} \& \ref{d:causal}) difference equations $u_{n+1} = A u_n + F(u)_n$ with $F(0) = 0$. Section 2 is devoted to basic properties of sequence spaces and existence of solutions. 

In this paper we follow an approach which goes back to \cite{Picard:McGhee2011} (see also \cite{Kalauch:Picard:Siegmund:Trostorff:Waurick2014, Picard:Trostorff:Waurick2014} and the references therein).

\section{Difference equations on sequence spaces}

\begin{lemma}[Exponentially weighted $\ell_p$ spaces]\label{lem:ell2rho}
Let  $1 \leq p < \infty$, $\rho >0$. Define
\begin{align*}
   \ell_{p,\rho}(\Z; H)
   &\coloneqq
   \setm{x \in H^\Z}{\sum_{k \in \Z} |x_k|_H^p \rho^{-p k} < \infty},
\\
   \ell_{\infty,\rho}(\Z; H)
   &\coloneqq
   \setm{x \in H^\Z}{\sup_{k \in \Z} |x_k|_H \rho^{-k} < \infty}.
\end{align*}
Then $\ell_{p,\rho}(\Z; H)$ and $\ell_{\infty,\rho}(\Z; H)$ are Banach spaces with norms
\begin{equation*}
   |x|_{\ell_{p,\rho}(\Z; H)} 
   \coloneqq
   \Big( \sum_{k \in \Z} |x_k|_H^p \rho^{-p k} \Big)^{\frac{1}{p}}
   \qquad
   (x \in \ell_{p,\rho}(\Z; H))
\end{equation*}
and
\begin{equation*}
   |x|_{\ell_{\infty,\rho}(\Z; H)}
   \coloneqq
   \sup_{k \in \Z} |x_k|_H \rho^{-k}
   \qquad
   (x \in \ell_{\infty,\rho}(\Z; H)),
\end{equation*}
respectively. Moreover, $\ell_{2,\rho}(\Z; H)$ is a Hilbert space with the inner product
\begin{equation*}
   \langle x, y \rangle_{\ell_{2,\rho}(\Z; H)}
   \coloneqq
   \sum_{k \in \Z} \big\langle x_k, y_k \big\rangle_{\!H} \rho^{-2 k}
   \qquad
   (x, y \in \ell_{2,\rho}(\Z; H)).
\end{equation*}
Write $\ell_{p}(\Z; H) \coloneqq \ell_{p,1}(\Z; H)$ for $1 \leq p \leq \infty$.
\end{lemma}

\begin{remark}[Intersections of weighted $\ell_p$ spaces]
Let $1 \leq p \leq \infty$, $0 < \rho_1 < \rho_2$. Then
\begin{equation*}
   \ell_{p,\rho_1}(\Z; H) \cap \ell_{p,\rho_2}(\Z; H)
   =
   \bigcap_{\rho \in [\rho_1, \rho_2]} \ell_{p,\rho}(\Z; H).
\end{equation*}
\end{remark}

\begin{lemma}[Natural embedding of one-sided sequences]\label{lem:one-sided}
Let $1 \leq p \leq \infty$, $\rho > 0$, $a \in \Z$. Define
\begin{equation*}
   \ell_{p,\rho}(\Z_{\geq a}; H)
   \coloneqq
   \setm{x {\restriction_{\Z_{\geq a}}}}{x \in \ell_{p,\rho}(\Z; H)},
\end{equation*}
and for $x \in H^{\Z_{\geq a}}$ define $\iota x \in H^{\Z}$ by
\begin{equation*}
   (\iota x)_k
   \coloneqq
   \begin{cases}
      0 & \text{if } k < a,
   \\
      x_k & \text{if } k \geq a.
   \end{cases}
\end{equation*}
Then $\ell_{p,\rho}(\Z_{\geq a}; H)$ is a Banach space with norm $|\cdot|_{\ell_{p,\rho}(\Z_{\geq a}; H)} \coloneqq  |\iota \cdot|_{\ell_{p,\rho}(\Z; H)}$, and
\begin{equation*}
   \iota \colon \ell_{p,\rho}(\Z_{\geq a}; H) \hookrightarrow \ell_{p,\rho}(\Z; H)
\end{equation*}
is an isometric embedding. Write $\ell_{p,\rho}(\Z_{\geq a}; H) \subseteq \ell_{p,\rho}(\Z; H)$.
\end{lemma}

\begin{lemma}[Scale of one-sided weighted $\ell_p$ spaces]\label{l:l1emb}
Let $1 \leq p < q \leq \infty$, $\rho, \eps > 0$, $a \in \Z$. Then

(a) $\ell_{p,\rho}(\Z_{\geq a}; H) \subsetneq \ell_{q,\rho}(\Z_{\geq a}; H)$,

(b) $\ell_{q,\rho}(\Z_{\geq a}; H) \subsetneq \ell_{p,\rho + \eps}(\Z_{\geq a}; H)$.
\end{lemma}

\begin{proof} (a): Let $x \in \ell_{p,\rho}(\Z_{\geq a}; H)$, i.e.\ $\sum_{k=a}^{\infty} |x_k|_H^p \rho^{-p k} < \infty$. Then $\exists k_0 \in \Z_{\geq a}, C \in [0,1) \colon$
\begin{equation*}
   |x_k|_H^p \rho^{-pk} \leq C
   \qquad (k \geq k_0).
\end{equation*}
Hence, $\sup_{k \geq k_0} |x_k|_H \rho^{-k} \leq C^{\frac{1}{p}}$, proving that $x \in \ell_{\infty,\rho}(\Z_{\geq a}; H)$. If $q < \infty$, then $(|x_k|_H \rho^{-k})^{q - p} < 1$ for $k \geq k_0$ and therefore
\begin{equation*}
   |x_k|_H^q \rho^{-kq}
   =
   |x_k|_H^p \rho^{-kp} (|x_k|_H \rho^{-k})^{q - p}
   <
   |x_k|_H^p \rho^{-kp}
   \qquad (k \geq k_0),
\end{equation*}
proving that $x \in \ell_{q,\rho}(\Z_{\geq a}; H)$.

(b): Let $x \in \ell_{q,\rho}(\Z_{\geq a}; H)$. Show that $x \in \ell_{1,\rho + \eps}(\Z_{\geq a}; H)$. By (a), $x \in \ell_{p,\rho + \eps}(\Z_{\geq a}; H)$.
If $q = \infty$, using the fact that $|x_k|_H \rho^{-k} \leq |x|_{\ell_{\infty,\rho}(\Z_{\geq a}; H)}$ for $k \in \Z_{\geq a}$,
\begin{equation*}
   \sum_{k=a}^{\infty} |x_k|_H (\rho + \eps)^{-k}
   =
   \sum_{k=a}^{\infty} |x_k|_H \rho^{-k} (\tfrac{\rho}{\rho + \eps})^{k}
   \leq
   |x|_{\ell_{\infty,\rho}(\Z_{\geq a}; H)}
   \big( \tfrac{\rho}{\rho + \eps} \big)^{a}
   \big( \tfrac{\rho + \eps}{\eps} \big),
\end{equation*}
proving that $x \in \ell_{1,\rho + \eps}(\Z_{\geq a}; H)$.
If $q < \infty$, let $q' \in (0,\infty)$ with $\frac{1}{q} + \frac{1}{q'} = 1$. Using H\"{o}lder's inequality,
\begin{align*}
   \sum_{k=a}^{\infty} |x_k|_H (\rho + \eps)^{-k}
   &=
   \sum_{k=a}^{\infty} |x_k|_H \rho^{-k} (\tfrac{\rho}{\rho + \eps})^{k}
   \leq
   \Big( \sum_{k=a}^\infty
   |x_k|_H^q \rho^{-kq} 
   \Big)^{\frac{1}{q}}
   \Big( \sum_{k=a}^\infty
   \big( \tfrac{\rho}{\rho + \eps} \big)^{kq'}
   \Big)^{\frac{1}{q'}}
\\
   &=
   |x|_{\ell_{q, \rho}(\Z_{\geq a}; H)}
   \big( \tfrac{\rho}{\rho + \eps} \big)^{aq'}
   \big( 1 - \big( \tfrac{\rho}{\rho + \eps} \big)^{q'} \big)^{-\frac{1}{q'}},
\end{align*}
proving that $x \in \ell_{1,\rho + \eps}(\Z_{\geq a}; H)$.
\end{proof}

\begin{example}{Solution to scalar equation with Kronecker delta impulse}
Let $x \in \C$, $a \in \C\setminus \{0\}$.
We consider the scalar difference equation
\begin{equation*}
   u_{n+1} = a u_n + \delta_{-1,n} x
   \qquad (n \in \Z).
\end{equation*}
A sequence $u \in \C^\Z$ satisfies the preceding difference equation if and only if
\begin{equation*}
   u_{n} 
   = 
   \begin{cases}
      a^{n + 1} c & \text{if } n \in \Z_{\leq -1},
   \\
      a^{n + 1} c + a^n x  & \text{if } n \in \Z_{\geq 0},
   \end{cases}
\end{equation*}
for some $c \in \C$.
In dependence on $c$ the support of $u$ satisfies
\begin{align*}
    &\operatorname{supp} u
    \subseteq
    \begin{cases}
        \Z_{\leq -1} & \text{if } c = -x/a,
        \\
        \Z_{\geq 0} & \text{if } c = 0,
    \end{cases}
\\
   \intertext{and}
\\
   &\operatorname{supp} u = \Z \quad \text{if } c \neq 0, c \neq -x/a.
\end{align*}
Let $\rho > 0$ and let $1 \leq p \leq \infty$.
If $\rho \neq |a|$, we can select a unique solution by requiring that $u \in \ell_{p,\rho}(\Z; \C)$.
Indeed if $\rho \neq |a|$ then
\begin{align*}
   u \in \ell_{p,\rho}(\Z; \C)
   &\quad\Leftrightarrow\quad
   \begin{cases}
      (\rho < |a| \;\wedge\; \operatorname{supp} u \subseteq \Z_{\leq -1}) \;\vee
   \\
      (\rho > |a| \;\wedge\; \operatorname{supp} u \subseteq \Z_{\geq 0})
   \end{cases}
   \\
   &\quad\Leftrightarrow\quad
   \begin{cases}
      (\rho < |a| \;\wedge\; c = -x/a) \;\vee
   \\
      (\rho > |a| \;\wedge\; c = 0).
   \end{cases}
\end{align*}
\end{example}

\pagebreak[4]
\begin{lemma}[Shift operator]\label{lem:Shift}
Let $1 \leq p \leq \infty$, $\rho > 0$. Then
\begin{align*}
   \tau \colon \ell_{p,\rho}(\Z; H) & \rightarrow \ell_{p,\rho}(\Z; H),
\\
   (x_k)_{k \in \Z} & \mapsto (x_{k+1})_{k \in \Z},
\end{align*}
is linear, bounded, invertible, and 
\begin{equation*}
   \|\tau^n\|_{L(\ell_{p,\rho}(\Z; H))} = \rho^n
   \qquad
   (n \in \Z).
\end{equation*}
\end{lemma}

\begin{proof}
Let $1 \leq p \leq \infty$, $\rho > 0$, $x \in \ell_{p,\rho}(\Z; H)$, $n \in \Z$. Using the fact that
\begin{equation*}
   (\tau^n x)_{k} = x_{n + k}
   \qquad (k \in \Z),
\end{equation*}
if $p < \infty$, 
\begin{equation*}
   |\tau^n x|_{\ell_{p,\rho}(\Z; H)}^p
   =
   \sum_{k \in \Z} |x_{n + k}|_H^p \rho^{-pk}
   =
   \sum_{k \in \Z} |x_{n + k}|_H^p \rho^{-p(n + k)} \rho^{pn}
   =
   \rho^{pn} |x|_{\ell_{p,\rho}(\Z; H)}^p,
\end{equation*}
proving that $|\tau^n x|_{\ell_{p,\rho}(\Z; H)} = \rho^{n} |x|_{\ell_{p,\rho}(\Z; H)}$, and, if $p = \infty$,
\begin{equation*}
   |\tau^n x|_{\ell_{\infty,\rho}(\Z; H)}
   =
   \sup_{k \in \Z} |x_{n + k}|_H \rho^{-k}
   =
   \sup_{k \in \Z} |x_{n + k}|_H \rho^{-(n + k)} \rho^{n}
   =
   \rho^n |x|_{\ell_{\infty,\rho}(\Z; H)}.
   \qedhere
\end{equation*}
\end{proof}

\begin{theorem}[Existence and uniqueness of solutions]\label{t:existence1}
Let $1 \leq p \leq \infty$, $\rho > 0$, $F \colon \ell_{p,\rho}(\Z; H) \rightarrow \ell_{p,\rho}(\Z; H)$ with
\begin{equation*}
   |F|_{\operatorname{Lip}(\ell_{p,\rho}(\Z; H))}
   \coloneqq
   \sup_{x, y \in \ell_{p,\rho}(\Z; H), x \neq y}
   \frac{|F(x) - F(y)|_{\ell_{p,\rho}(\Z; H)}}{|x-y|_{\ell_{p,\rho}(\Z; H)}}
   < \rho.
\end{equation*}
Then the difference equation
\begin{equation*}
   \tau u = F(u)
   \qquad
   (u \in \ell_{p,\rho}(\Z; H))
\end{equation*}
has a unique solution.
\end{theorem}

\begin{proof}
Then
\begin{align*}
   \mathcal{T} \colon \ell_{p,\rho}(\Z; H) &\mapsto \ell_{p,\rho}(\Z; H)
\\
   u & \mapsto \tau^{-1} F(u)
\end{align*}
is a contraction, since, by Lemma \ref{lem:Shift}, for $x, y \in \ell_{p,\rho}(\Z; H)$,
\begin{equation*}
   |\mathcal{T}(x) - \mathcal{T}(y)|
   \leq
   |\tau^{-1}|_{L(\ell_{p,\rho}(\Z; H))} |F(x) - F(y)|_{\ell_{p,\rho}(\Z; H)}
   = 
   \frac{1}{\rho} |F|_{\operatorname{Lip}(\ell_{p,\rho}(\Z; H))}
   < 1.
\end{equation*}
Its unique fixed point $u = \mathcal{T}(u)$ satisfies $\tau u = F(u)$.
\end{proof}

Explicit difference equations \cite{Elaydi2005} of the form $u_{n+1} =  G(u_n)$ on $\Z_{\geq 0}$ with initial condition $u_0 \in \R^n$ have a unique solution $u \colon \Z_{\geq 0} \to \R^n$ for arbitrary functions $G \colon \R^n \to \R^n$. In the following theorem, solutions of more general initial value problems are characterized as one-sided sequences which solve a difference equation on $\Z$ with an impulse corresponding to the initial condition.

\begin{theorem}[Initial value problems]\label{t:initial_value_problem}
Let $A \in L(H)$, $F \colon H^\Z \to H^\Z$, $u \in H^\Z$, $\operatorname{spt} u \subseteq \Z_{\geq 0}$, $\operatorname{spt} F(u) \subseteq \Z_{\geq 0}$, $x \in H$. Then the following statements are equivalent.
\begin{enumerate}
   \item[(i)] $   u_{n+1} = A u_n + F(u)_n, \quad u_0 = x \qquad (n \in \Z_{\geq 0})$. 
      \hfill\text{(Initial value problem)}
   \item[(ii)] $u_n = A^n x + \sum_{k=0}^{n-1} A^{n-k-1} F(u)_k \qquad (n \in \Z_{\geq 0})$. 
      \hfill\text{(Variation of constants)}
   \item[(iii)] $\tau u = Au + F(u) + \delta_{-1} x$.
         \hfill\text{(Initial condition as impulse)}
\end{enumerate}
\end{theorem}

\begin{proof}
(i) $\Rightarrow$ (ii): By induction for $n \in \Z_{\geq 0}$, $u_n = A^n x + \sum_{k=0}^{n-1} A^{n-k-1} F(u)_k$, since
\begin{align*}
   u_{n+1} 
   &=
   A u_n + F(u)_n
   =
   A \Big(A^n x + \sum_{k=0}^{n-1} A^{n-k-1} F(u)_k \Big) + F(u)_n 
   \\
   &=
   A^{n+1} x + \sum_{k=0}^{n} A^{n-k} F(u)_k.
\end{align*}

(ii) $\Rightarrow$ (iii): $u_0 = x$, and for $n \in \Z_{\geq 0}$, $u_{n+1} = A u_n + F(u)_n$, since
\begin{align*}
   u_{n+1} 
   &= 
   A^{n+1} x + \sum_{k=0}^{n} A^{n-k} F(u)_k
   =   
   A \Big(A^n x + \sum_{k=0}^{n-1} A^{n-k-1} F(u)_k \Big) + F(u)_n 
   \\
   &=
   A u_n + F(u)_n.
\end{align*}
Using $\operatorname{spt} u, \operatorname{spt} F(u) \subseteq \Z_{\geq 0}$, for $n \in \Z$,
\begin{equation*}
   (Au + F(u) + \delta_{-1} x)_n
   = 
   A u_n + F(u)_n + \delta_{-1,n} x
   =
   \begin{cases}
      0 
      & \text{if } n \leq -2,
   \\
      x
      & \text{if } n = -1,
   \\
      Au_n + F(u)_n
      & \text{if } n \geq 0,
   \end{cases}
\end{equation*}
hence $(\tau u)_n = u_{n+1} = (Au + F(u) + \delta_{-1} x)_n$, proving (iii).

(iii) $\Rightarrow$ (i): For $n \in \Z$, $u_{n+1} = (\tau u)_n = (Au + F(u) + \delta_{-1} x)_n = A u_n + F(u)_n + \delta_{-1,n} x$.
Using the fact that $\operatorname{spt} u \subseteq \Z_{\geq 0}$, $\operatorname{spt} F(u) \subseteq \Z_{\geq 0}$, we see that $u_0 = A u_{-1} + F(u)_{-1} + x = x$, implying $(i)$.
\end{proof}

\begin{remark}[Classical initial value problems]
Let $G \colon \R^n \to \R^n$ and
\begin{align*}
   F(u)_n
   \coloneqq
   \begin{cases}
      G(u_n) & \text{if } n \in \Z_{\geq 0},
   \\
      0 & \text{if } n \in \Z_{\leq -1},
   \end{cases}
\end{align*}
for $u \colon \Z \to \R^n$.
Then every solution $u \colon \Z \to \R^n$ of $\tau u = Au + F(u) + \delta_{-1} x$ with $\operatorname{spt} u \subseteq \Z_{\geq 0}$ gives rise to a solution of the initial value problem $u_{n+1} =  G(u_n)$ on $\Z_{\geq 0}$ with initial condition $u_0 = x \in \R^n$ and vice versa.
\end{remark}

\section{The $\cZ$ transform and the operator $(\tau - A)^{-1}$}

Let $A \in L(H)$, $\rho>0$, $F: \ell_{2, \rho}(\Z; H) \rightarrow \ell_{2, \rho}(\Z; H)$. As in Theorem \ref{t:initial_value_problem}, for $x \in H$, consider initial value problems of the form
\begin{equation*}
    \tau u = A u + F(u) + \delta_{-1} x
    \qquad
    (u \in \ell_{2, \rho}(\Z; H)).
\end{equation*}
In this section we reformulate this problem equivalently as a fixed point problem
\begin{equation*}
    u = (\tau - A)^{-1}(F(u) + \delta_{-1} x)
\end{equation*}
by discussing the operator $(\tau - A)^{-1}$ on $\ell_{2, \rho}(\Z; H)$. Our main tool is the $\cZ$ transform.

\begin{lemma}[$L_2$ space on circle and orthonormal basis]
Let $\rho >0$, $S_\rho \coloneqq \setm{z \in \C}{|z| = \rho}$. Define
\begin{equation*}
   L_{2}(S_\rho; H)
   \coloneqq
   \setm{f \colon S_\rho \rightarrow H}{ \int_{S_\rho} | f(z) |_H^2 \frac{\mathrm{d}z}{|z|} < \infty}.
\end{equation*}
Then $L_{2}(S_\rho; H)$ is a Hilbert space with the inner product
\begin{equation*}
   \langle f, g \rangle_{L_{2}(S_\rho; H)}
   \coloneqq
   \frac{1}{2\pi}
   \int_{S_\rho} \big\langle f(z), g(z) \big\rangle_{\!H} \frac{\mathrm{d}z}{|z|}
   \qquad
   (f, g \in L_{2}(S_\rho; H)).
\end{equation*}
Moreover, let $(\psi_n)_{n \in \Z}$ be an orthonormal basis in $H$.
Then $(p_{k,n})_{k, n \in \Z}$ with
\begin{equation*}
   p_{k,n}(z) \coloneqq \rho^{k} z^{-k} \psi_n
   \qquad (z \in S_\rho)
\end{equation*}
is an orthonormal basis in $L_2(S_\rho; H)$.
\end{lemma}

\begin{proof}
Let $L_2((-\pi, \pi); H)$ denote the Hilbert space of square integrable functions from $(-\pi, \pi)$ to $H$ with the scalar product
\begin{equation*}
    \left<f, g\right>_{L_2((-\pi, \pi); H)} \coloneqq \frac{1}{2 \pi} \int_{-\pi}^\pi \left<f(z), g(z)\right>_H dz
    \qquad
    (f, g \in L_2((-\pi, \pi); H)).
\end{equation*}
The isometric isomorphism
\begin{equation*}
    \Phi: L_2(S_\rho; H) \rightarrow L_2((-\pi, \pi); H), \quad f \mapsto \big(\phi \mapsto f(\rho e^{i \phi})\big),
\end{equation*}
maps $(p_{k, n})_{k, n \in \Z}$ onto the orthonormal basis $(\phi \mapsto e^{-i k \phi} \psi_n)_{k, n \in \Z}$ of $L_2((-\pi, \pi); H)$, proving that $(\Phi(p_{k, n}))_{k, n \in \Z}$ is an orthonormal basis of $L_2(S_\rho; H)$.
\end{proof}

\pagebreak[4]
\begin{theorem}[$\cZ$ transform]\label{t:ztransform}
Let $\rho >0$.
The operator
\begin{align*}
   \cZ_\rho \colon \ell_{2,\rho}(\Z; H) &\rightarrow L_2(S_\rho; H),
\\
   x & \mapsto \Big( z \mapsto \sum_{k \in \Z} \langle \psi_n, \rho^{-k} x_k\rangle_{\!H} p_{k, n}(z) \Big),
\end{align*}
is well-defined and unitary.
For $x \in \ell_{1, \rho}(\Z; H) \subseteq \ell_{2, \rho}(\Z; H)$ we have
\begin{equation*}
    \cZ_\rho(x) = \left(z \mapsto \sum_{k \in \Z} x_k z^{-k}\right).
\end{equation*}
\end{theorem}

\begin{remark}[$\cZ$ transform of $x \in \ell_{2,\rho}(\Z; H) \setminus \ell_{1,\rho}(\Z; H)$]\label{rem:z-transform}
Let $\rho > 0$, $x \in \ell_{2,\rho}(\Z; H) \setminus \ell_{1,\rho}(\Z; H)$. Then
\begin{equation*}
   \sum_{k \in \Z} x_k z^{-k} 
\end{equation*}
does not necessarily converge for all $z \in S_\rho$.
For example if $H = \C$, $x \in \ell_{2, \rho}(\Z; H) \setminus \ell_{1, \rho}(\Z; H)$ with $x_k \coloneqq \frac{\rho^k}{k}$ and $z = \rho$.
\end{remark}

\begin{proof}[Proof of Theorem \ref{t:ztransform}]
Let $\rho >0$, $x \in \ell_{2,\rho}(\Z; H)$.
We prove that $\cZ_\rho$ is well-defined, i.e. that $\cZ_\rho(x)$ exists as an element in $L_2(S_\rho; H)$.
We compute
\begin{align*}
   \big| 
   \sum_{k, n \in \Z} \langle \psi_n, \rho^{-k} x_k\rangle_{\!H} p_{k,n} 
   \big|_{L_2(S_\rho; H)}^2
   &=
   \frac{1}{2\pi}
   \int_{S_\rho} \big| 
   \sum_{k, n \in \Z} \langle \psi_n, \rho^{-k} x_k\rangle_{\!H} p_{k,n}(z)
   \big|_{\!H}^2
    \frac{\mathrm{d}z}{|z|} 
\\
   &=
   \frac{1}{2\pi}
   \int_{S_\rho} \big| 
   \sum_{k, n \in \Z} \langle \psi_n, \rho^{-k} x_k\rangle_{\!H} \rho^k z^{-k} \psi_n
   \big|_{\!H}^2
    \frac{\mathrm{d}z}{|z|} 
\\
   &=
   \frac{1}{2\pi}
   \int_{S_\rho} \big| 
   \sum_{k, n \in \Z} \langle \psi_n, x_k\rangle_{\!H} z^{-k} \psi_n
   \big|_{\!H}^2
    \frac{\mathrm{d}z}{|z|} 
\\
   &=
   \frac{1}{2\pi}
   \int_{-\pi}^{\pi} \big| 
   \sum_{k, n \in \Z} \langle \psi_n, x_k\rangle_{\!H} (\rho e^{i \phi})^{-k} \psi_n
   \big|_{\!H}^2
   \,\mathrm{d}\phi 
\\
   &=
   \frac{1}{2\pi}
   \int_{-\pi}^{\pi} \big| 
   \sum_{k \in \Z} \rho^{-k} e^{-i \phi k} 
   \sum_{n \in \Z}
   \langle \psi_n, x_k\rangle_{\!H} \psi_n
   \big|_{\!H}^2
   \,\mathrm{d}\phi 
\\
   &=
   \frac{1}{2\pi}
   \int_{-\pi}^{\pi} \big| 
   \sum_{k \in \Z} \rho^{-k} e^{-i \phi k} x_k
   \big|_{\!H}^2
   \,\mathrm{d}\phi 
\\
   &=
   \frac{1}{2\pi}
   \int_{-\pi}^{\pi} \big\langle 
   \sum_{k \in \Z} \rho^{-k} e^{-i \phi k} x_k,
   \sum_{\ell \in \Z} \rho^{-\ell} e^{-i \phi \ell} x_\ell
   \big\rangle_{\!H}
   \,\mathrm{d}\phi 
\\
   &=
   \frac{1}{2\pi}
   \int_{-\pi}^{\pi} \sum_{k, \ell \in \Z} \rho^{-k - \ell} e^{-i \phi (\ell  - k)} \big\langle 
   x_k, x_\ell
   \big\rangle_{\!H}
   \,\mathrm{d}\phi 
\\
   &=
   \frac{1}{2\pi}
   \sum_{k, \ell \in \Z} 
   \rho^{-k - \ell} \big\langle x_k, x_\ell \big\rangle_{\!H}
   \int_{-\pi}^{\pi} 
   e^{-i \phi (\ell - k)} 
   \,\mathrm{d}\phi 
\\
   &=
   \sum_{k \in \Z} |x_k|_{H}^2 \rho^{-2 k}
\\
   &=
   |x|_{\ell_{2,\rho}(\Z; H)}^2,
\end{align*}
proving that $\cZ_\rho$ is well-defined and an isometry.

To prove that $\cZ_\rho$ is unitary, it remains to show that $\cZ_\rho$ is surjective. 
Let $f \in L_2(S_\rho; H)$. Define
\begin{equation*}
   x_k
   \coloneqq
   \frac{1}{2\pi}
   \int_{S_\rho}
   f(z) \rho^{2k}\overline{z}^{-k} \,\frac{\mathrm{d}z}{|z|}
   \qquad 
   (k \in \Z).
\end{equation*}
Then $\cZ_\rho x = f$, since for $k, n \in \Z$,
\begin{align*}
   \langle \cZ_\rho x, p_{k,n} \rangle_{L_2(S_\rho; H)}
   &=
   \frac{1}{2\pi}
   \int_{S_\rho} 
   \big\langle \sum_{\ell, m \in \Z} \langle \psi_m, \rho^{-\ell} x_\ell \rangle_{\!H} p_{\ell,m}(z), 
   p_{k,n}(z) \big\rangle_{\!H}
   \frac{\mathrm{d}z}{|z|}
\\
   &=
   \frac{1}{2\pi}
   \int_{S_\rho} 
   \sum_{\ell, m \in \Z} \langle \rho^{-\ell} x_\ell, \psi_m \rangle_{\!H}
   \big\langle p_{\ell,m}(z), p_{k,n}(z) \big\rangle_{\!H}
   \frac{\mathrm{d}z}{|z|}
\\
   &=
   \sum_{\ell, m \in \Z} \langle \rho^{-\ell} x_\ell, \psi_m \rangle_{\!H}
   \frac{1}{2\pi}
   \int_{S_\rho} 
   \big\langle p_{\ell,m}(z), p_{k,n}(z) \big\rangle_{\!H}
   \frac{\mathrm{d}z}{|z|}
\\
   &=
   \langle \rho^{-k} x_k, \psi_n \rangle_{\!H}
\\
   &=
   \Big\langle  
   \frac{1}{2\pi}
   \int_{S_\rho} f(z) \rho^{2k} (\overline{z})^{-k} \frac{\mathrm{d}z}{|z|}, 
   \rho^{-k} \psi_n 
   \Big\rangle_{\!H}
\\
&=
    \frac{1}{2\pi}
   \int_{S_\rho}\langle f(z)   , 
   \rho^{k}{z}^{-k} \psi_n 
   \rangle_{\!H}\frac{\mathrm{d}z}{|z|}
\\
   &= 
   \frac{1}{2\pi}
   \int_{S_\rho} \langle f(z)  , 
    p_{k,n}({z})\rangle_{\!H} \frac{\mathrm{d}z}{|z|}
\\
   &=
   \big\langle 
   f, p_{k,n} \big\rangle_{L_2(S_\rho; H)}.
\end{align*}

For $x \in \ell_{2,\rho}(\Z; H) \cap \ell_{1,\rho}(\Z; H)$ and $z \in S_\rho$ we compute
\begin{align*}
   \sum_{k, n \in \Z} \langle \psi_n, \rho^{-k} x_k \rangle_{\!H} p_{k,n}(z)
   &= 
   \sum_{k, n \in \Z} \langle \psi_n, \rho^{-k} x_k \rangle_{\!H} \rho^k z^{-k} \psi_n 
\\
   &=
   \sum_{k, n \in \Z} \langle \psi_n, x_k z^{-k} \rangle_{\!H} \psi_n 
\\
   &=
   \sum_{k \in \Z} x_k z^{-k}.
   \qedhere
\end{align*}
\end{proof}
\smallskip

\begin{lemma}[Shift is unitarily equivalent to multiplication]\label{lem:mo} Let $\rho > 0$. Then 
\[
   \mathcal{Z}_\rho \tau \mathcal{Z}_\rho^* = \mathrm{m},
\]
where $\mathrm{m}$ is the multiplication-by-the-argument operator acting in $L_2(S_\rho;H)$, i.e.,
\begin{align*}
   \mathrm{m}\colon L_2(S_\rho;H) &\to L_2(S_\rho;H)\\
   f&\mapsto (z\mapsto zf(z)).
\end{align*} 
\end{lemma}
\begin{proof}
By the boundedness of the operators involved and the unitarity of $\mathcal{Z}_\rho$, it suffices to prove
 \[
    \mathrm{m} \mathcal{Z}_\rho x = \mathcal{Z}_\rho\tau x
 \]
for $x\in H^{\mathbb{Z}}$ with compact support. For this, we compute
\begin{align*}
  (\mathrm{m} \mathcal{Z}_\rho x)(z)& =z \mathcal{Z}_\rho x (z)
  \\ &=  z \sum_{k \in \Z} x_k z^{-k}
  \\ &=   \sum_{k \in \Z} x_k z^{-k+1}
  \\ &=   \sum_{k \in \Z} x_k z^{-(k-1)}=\sum_{k \in \Z} (\tau x)_k z^{-k}=(\mathcal{Z}_\rho \tau x)(z).
\end{align*}
\end{proof}

Next, we present a Payley--Wiener type result for the $\mathcal{Z}$ transform. 

\begin{lemma}[Characterization of positive support]\label{lem:positive-support}
Let $\rho > 0$, $x \in \ell_{2,\rho}(\Z; H)$. Then the following statements are equivalent:

(i) $\operatorname{spt} x \subseteq \Z_{\geq 0}$,

(ii) $z \mapsto \sum_{k \in \Z} x_k z^{-k}$ is analytic on $\C_{|\cdot| > \rho}$ and
\begin{equation}\label{e:Hardybound}
   \sup_{\mu > \rho} \int_{S_\mu} \big|\sum_{k \in \Z} x_k z^{-k}\big|_H^2 \,\frac{\mathrm{d}z}{|z|} < \infty.
\end{equation}
\end{lemma}

\begin{proof}
(i) $\Rightarrow$ (ii): Let $\eps > 0$. Then for $z \in \C$ with $|z| \geq \rho + \eps$ and $k \in \Z_{\geq 0}$, $|z|^{-k} \leq (\rho + \eps)^{-k}$, and
\begin{align*}
   \sum_{k = 0}^\infty |x_k z^{-k}|_H 
   &=
   \sum_{k = 0}^\infty |x_k|_H |z^{-k}| 
\\
   &= 
   \sum_{k = 0}^\infty |x_k|_H \rho^{-k} \rho^k |z^{-k}|
\\
   &\leq
   \sum_{k = 0}^\infty |x_k|_H \rho^{-k} \rho^k (\rho + \eps)^{-k}
\\
   &\leq
   \Big( \sum_{k = 0}^\infty |x_k|_H^2 \rho^{-2k} \Big)^{\frac{1}{2}} 
   \Big( \sum_{k = 0}^\infty \big(\tfrac{\rho}{\rho + \eps}\big)^{2k} \Big)^{\frac{1}{2}}
\\
   &=
   |x|_{\ell_{2,\rho}(\Z; H)} \big( 1 - (\tfrac{\rho}{\rho + \eps})^2 \big)^{-\frac{1}{2}}.
\end{align*}
Hence $z \mapsto \sum_{k \in \Z} x_k z^{-k}$ is analytic on $\C_{|\cdot|>\rho}$.
Moreover, for $\mu > \rho$, 
\begin{align*}
   |x|_{\ell_{2,\mu}(\Z; H)}^2
   &=
   \sum_{k=0}^{\infty} |x_k|_H^2 \mu^{-2k}
\\
   &\leq
   \sum_{k=0}^{\infty} |x_k|_H^2 \rho^{-2k}
\\
   &=
   |x|_{\ell_{2,\rho}(\Z; H)}^2.
\end{align*}
Using the fact that $\cZ_\mu$ is unitary,
\begin{equation*}
   \frac{1}{2 \pi} \int_{S_\mu} \big|\sum_{k \in \Z} x_k z^{-k}\big|_H^2 \,\frac{\mathrm{d}z}{|z|}
   =
   | \cZ_\mu x |_{L_2(S_\mu; H)}^2
   =
   | x |_{\ell_{2,\mu}(\Z; H)}^2
   \leq
   | x |_{\ell_{2,\rho}(\Z; H)}^2,
\end{equation*}
proving \eqref{e:Hardybound}.

(ii) $\Rightarrow$ (i): Let $k \in \Z_{< 0}$. Define $e^{(k)} \in \C^{\Z}$ by $e^{(k)}_\ell \coloneqq \delta_{k,\ell}$ for $\ell \in \Z$, then $e^{(k)} \in \ell_{2}(\Z;\C)$. Let $n \in \Z$. Then
\begin{align*}
   |\langle x_k, \psi_n \rangle_{\!H}|
   &=
   \big| \sum_{\ell \in \Z}
   \langle x_\ell, \psi_n \rangle_{\!H} e^{(k)}_\ell \big|
\\
   &=
   \big| \sum_{\ell \in \Z}
   \langle x_\ell, \mu^{2 \ell} \psi_n e^{(k)}_\ell \rangle_{\!H} \mu^{- 2 \ell} \big|
\\
   &=
   |\langle x, \mu^{2 \cdot} \psi_n e^{(k)}_{\cdot} \rangle_{\ell_{2, \mu}(\Z; H)}|
\\
   &=
   |\langle \cZ_\mu x, \cZ_\mu \mu^{2 \cdot} \psi_n e^{(k)}_{\cdot} \rangle_{L_2(S_\mu; H)}|
\\
   &\leq
   | \cZ_\mu x |_{L_2(S_\mu; H)} \,
   | \cZ_\mu \mu^{2 \cdot} \psi_n e^{(k)}_{\cdot} |_{L_2(S_\mu; H)}
\\
   &\leq
   | \cZ_\mu x |_{L_2(S_\mu; H)} \,
   \mu^{k} |\psi_n|.
\end{align*}

Since $(\cZ_\mu \mu^{2 \cdot} \psi_n e^{(k)}_{\cdot})(z)
=
\sum_{\ell \in \Z} \mu^{2 \ell} \psi_n e^{(k)}_\ell z^{-\ell}
=
\mu^{2 k} \psi_n z^{-k}$ for $z \in S_\mu$,
\begin{align*}
   |\langle x_k, \psi_n \rangle_{\!H}|
   &\leq
   | \cZ_\mu x |_{L_2(S_\mu; H)} \,
   \mu^{k} |\psi_n|.
\end{align*}
By \eqref{e:Hardybound}, that is, $\sup_{\mu>\rho}|\mathcal{Z}_\mu x|_{L_2(S_\mu;H)}<\infty$, and since $k<0$, the right-hand side tends to $0$ as $\mu \to \infty$, proving that $x_k = 0$.
\end{proof}

\begin{definition}[Causal linear operator]\label{d:linearcausal}
    We call a linear operator $B  \colon \ell_{2,\rho}(\Z; H) \rightarrow \ell_{2,\rho}(\Z; H)$ \emph{causal}, if for all $a\in \Z$, $f\in \ell_{2,\rho}(\Z;H)$, we have
    \begin{equation*}
        \operatorname{spt} f \subseteq \mathbb{Z}_{\geq a} 
        \quad\Rightarrow\quad
        \operatorname{spt} Bf \subseteq\mathbb{Z}_{\geq a}.
    \end{equation*}
\end{definition}

Note, that if an operator $B  \colon \ell_{2,\rho}(\Z; H) \rightarrow \ell_{2,\rho}(\Z; H)$ commutes with the shift operator $\tau$, the operator $B$ is causal, if and only if for all $f \in \ell_{2, \rho}(\Z; H)$ we have
\begin{equation*}
    \operatorname{spt} f \subseteq \mathbb{Z}_{\geq 0}
    \quad\Rightarrow\quad
    \operatorname{spt} Bf \subseteq \mathbb{Z}_{\geq 0}.
\end{equation*}
The next major result is Theorem \ref{t:causal} which characterizes the causality of the operator $(\tau - A)^{-1}$ on $\ell_{2, \rho}(\Z; H)$ by the spectral radius of $A$.
Recall \cite[VIII.3.6, p.\ 222]{Katznelson2004} that for $A \in L(H)$ with spectrum $\sigma(A)$, the spectral radius
\begin{equation*}
   r(A)
   \coloneqq
   \sup \setm{|z|}{z \in \sigma(A)}
\end{equation*}
of $A$ satisfies
\begin{equation*}
   r(A)
   =
   \lim_{n \to \infty} |A^n|_{L(H)}^{1/n}.
\end{equation*}
Let $A\in L(H)$ and $\rho > 0$.
We denote the operators $\ell_{2,\rho}(\Z, H)\rightarrow \ell_{2,\rho}(\Z, H)$, $x\mapsto (Ax_k)_k$,
        and $L_2(S_\rho, H)\rightarrow L_2(S_\rho, H)$, $f\mapsto (z \mapsto Af(z))$,
        which have the same operator norm as $A$, again by $A$.

\begin{theorem}[Characterization of causality of $(\tau - A)^{-1}$ by spectral radius]\label{t:causal}
Let $\rho > 0$, $A \in L(H)$. Then the following statements are equivalent:

(i) $(\tau - A)^{-1} \in L(\ell_{2,\rho}(\Z; H))$ is causal,

(ii) $(\tau - A)^{-1} \in L(\ell_{2,\rho}(\Z; H))$ and $\operatorname{spt} (\tau - A)^{-1} \delta_{-1} x \subseteq \mathbb{Z}_{\geq 0}$ for all $x\in H$,

(iii) $\rho > r(A)$.
\end{theorem}

\begin{proof}
(i) $\Rightarrow$ (ii):
By assumption $(\tau - A)^{-1} \in L(\ell_{2,\rho}(\Z; H))$.
Let $u \coloneqq (\tau - A)^{-1} \delta_{-1} x$.
By causality of $(\tau - A)^{-1}$, it follows that $\operatorname{spt} u \subseteq \Z_{\geq -1}$.
Using the fact that $\tau u = A u + \delta_{-1} x$, we get
\begin{equation*}
    u_{-1} = A u_{-2} + \delta_{-1, -2} x = 0,
\end{equation*}
proving $\operatorname{spt} u \subseteq \Z_{\geq 0}$.

(ii) $\Rightarrow$ (iii):
We have $(\tau - A)^{-1} \in L(\ell_{2,\rho}(\Z; H))$.
Hence, by Lemma \ref{lem:mo}, $(\mathrm{m} - A)^{-1} \in L(L_2(S_\rho; H))$.
By the definition of $\mathrm{m}$, this implies that $S_\rho$ along with a neighborhood of $S_\rho$ is contained in the resolvent set of $A \in L(H)$.
Therefore the spectral radius cannot be equal to $\rho$.
Assume $\rho < r(A)$.
By the formula for the spectral radius, there exists $\eps > 0$, $n_0 \in \N$ with $(\rho + \eps)^n \leq |A^n|_{L(H)}$ for all $n\geq n_0$.
Hence,
\begin{equation*}
   \Big|
   \frac{A^n}{\rho^n}
   \Big|_{L(H)}
   \to \infty
   \qquad (n \to \infty).
\end{equation*}
By the uniform boundedness principle there exists $x \in H$ with 
\begin{equation*}
   \frac{|A^n x|_H}{\rho^n |x|_H}
   \to \infty
   \qquad (n \to \infty),
\end{equation*}
i.e., there exists $n_1 \in \N$ such that $|A^n x|_H \geq \rho^n |x|_H$ for all $n \geq n_1$.
By hypothesis we obtain $u \coloneqq (\tau - A)^{-1} \delta_{-1} x \in \ell_{2, \rho}(\Z; H)$ and 
\begin{equation*}
      \operatorname{spt} u \subseteq \Z_{\geq 0}.
\end{equation*}
Since $\tau u = A u + \delta_{-1} x$, it follows that
\begin{equation*}
   u_n =
   \begin{cases}
      0 & \text{if } n \leq -1,
   \\
      A^n x & \text{if } n \geq 0.
   \end{cases}
\end{equation*}
Consequently, we obtain
\begin{align*}
   |u|_{\ell_{2, \rho}(\Z; H)}^2
   &=
   \sum_{k \in \Z} |u_k|_H^2 \rho^{-2 k}
\\
   &=
   \sum_{k = 0}^\infty |A^k x|_H^2 \rho^{-2 k}
\\
   &\geq
   \sum_{k = n_1}^\infty |x|_H \rho^{2 k} \rho^{-2 k}
   = \infty,
\end{align*}
which is a contradiction. Thus,  $\rho > r(A)$.

(iii) $\Rightarrow$ (i):
Let $\rho > r(A)$. Then $(\mathrm{m} - A)^{-1} \in L(L_2(S_\rho; H))$ and hence $(\tau - A)^{-1} \in L(\ell_{2,\rho}(\Z; H))$. To show (i), let $f \in \ell_{2, \rho}(\Z; H)$, $\operatorname{spt}f \subseteq \Z_{\geq 0}$, and prove that $u \coloneqq (\tau - A)^{-1}f$ satisfies $\operatorname{spt}u \subseteq \Z_{\geq 0}$.

By Lemma \ref{lem:positive-support}, $z \mapsto \sum_{k \in \Z} f_k z^{-k}$ is analytic on $\C_{|\cdot|>\rho}$. Hence, also
\begin{equation*}
   z \mapsto \sum_{k \in \Z} u_k z^{-k}
   =
   (z - A)^{-1} \sum_{k \in \Z} f_k z^{-k}
   \qquad 
   \text{is analytic on } \C_{|\cdot|>\rho}.
\end{equation*}
Next, by Lemma \ref{lem:positive-support},
$\sup_{\mu > \rho} \int_{S_\mu} \big|\sum_{k \in \Z} f_k z^{-k}\big|_H^2 \,\frac{\mathrm{d}z}{|z|} < \infty$. Moreover, we have    
\[  \sup_{z \in \C_{|\cdot|>\rho}} \big|(z - A)^{-1} \big|_{L(H)}<\infty.\] Indeed, this follows from $(z - A)^{-1} = \frac{1}{z} (1 - z^{-1} A)^{-1} = \frac{1}{z} \sum_{k=0}^\infty (\frac{1}{z} A)^k$ for $z\in \C_{|\cdot|>\rho}$ and $\rho>r(A)$. For $\mu > \rho$
\begin{align*}
   \int_{S_\mu} \big|\sum_{k \in \Z} u_k z^{-k}\big|_H^2 \,\frac{\mathrm{d}z}{|z|}
   &=
   \int_{S_\mu} \big|(z - A)^{-1} \sum_{k \in \Z} f_k z^{-k}\big|_H^2 \,\frac{\mathrm{d}z}{|z|}
\\
   &\leq
   \sup_{z \in \C_{|\cdot|>\rho}} \big|(z - A)^{-1} \big|_{L(H)}^2
   \int_{S_\mu} \big| \sum_{k \in \Z} f_k z^{-k}\big|_H^2 \,\frac{\mathrm{d}z}{|z|},
\end{align*}
hence $\sup_{\mu > \rho} \int_{S_\mu} \big|\sum_{k \in \Z} u_k z^{-k}\big|_H^2 \,\frac{\mathrm{d}z}{|z|} < \infty$. By Lemma \ref{lem:positive-support}, $\operatorname{spt}u \subseteq \Z_{\geq 0}$.
\end{proof}

Next, we address linear initial value problems. In particular, we show that for $x \in H$ the problem
\begin{equation*}
    \tau u = A u + \delta_{-1} x
\end{equation*}
has exactly one solution $u \in H^\Z$ satisfying $\operatorname{spt} u \subseteq \Z_{\geq 0}$
    and this solution $u$ solves the initial value problem
\begin{equation*}
    u_{n + 1} = A u_n, \quad u_0 = x
    \qquad
    (n \in \Z_{\geq 0}).
\end{equation*}

\begin{proposition}[Linear initial value problem]\label{p:linear-initial-value-problem}
Let $A \in L(H)$, $x \in H$ and $u \in H^\mathbb{Z}$.
Then the following statements are equivalent:
\begin{enumerate}
    \item[(i)] $\tau u = A u  + \delta_{-1}x \text{ and } \operatorname{spt}(u) \subseteq \Z_{\geq 0}$.

    \item[(ii)] $u_{n+1} = A u_n\text{ on }\mathbb{Z}_{>0} \text{ and }u_0=x 
       \text{ and }u_n =0 \text{ on }\mathbb{Z}_{< 0}$.

    \item[(iii)] $u_n =
        \begin{cases}
        0 & \text{if } n < 0,
        \\
        A^{n} x & \text{if } n \geq 0.
        \end{cases}$
\end{enumerate}
Moreover, if $\rho > r(A)$ then $u$ satisfying $(i), (ii), (iii)$ is unique in $\ell_{2, \rho}(\Z; H)$.
I.e., if $\rho > r(A)$ then $(i), (ii), (iii)$ are equivalent to $(iv)$ and $(v)$.
\begin{enumerate}
    \item[(iv)] $\tau u = A u  + \delta_{-1}x \text{ and } u \in \ell_{2, \rho}(\Z; H)$.

    \item[(v)] $u = (\tau - A)^{-1} \delta_{-1} x \in \ell_{2, \rho}(\Z; H) $.
\end{enumerate}
\end{proposition}
\begin{proof}
The equivalence of $(i), (ii)$ and $(iii)$ are reformulations in the spirit of Theorem \ref{t:initial_value_problem}.
By Theorem \ref{t:causal}, the operator $(\tau-A)^{-1}\in L(\ell_{2,\rho}(\mathbb{R};H))$ is well-defined and causal.
This shows $(iv) \Leftrightarrow (v)$ and $(v) \Leftrightarrow (i)$.
\end{proof}

We present a characterization of exponential stability next.

\begin{proposition}[Characterization of exponential stability of linear equations]\label{t:expstab}
Let $A \in L(H)$.
Then the following statements are equivalent:
\begin{enumerate}
    \item[(i)] $r(A) < 1$.
    
    \item[(ii)] There exists $\rho \in (0, 1)$ such that for all $x \in H$ the unique solution $u \in H^\Z$ of
    \begin{equation*}
        \tau u = A u + \delta_{-1} x
    \end{equation*}
    with $\operatorname{spt} u \subseteq \Z_{\geq 0}$ (c.f. Proposition \ref{p:linear-initial-value-problem} $(i) \Leftrightarrow (iii)$) satisfies $u \in \ell_{2, \rho}(\Z; H)$.
    
    \item[(iii)] There exists $\rho \in (0, 1)$ such that for all $x \in H$ for the unique solution $u \in H^\Z$ of
    \begin{equation*}
        \tau u = A u + \delta_{-1} x
    \end{equation*}
    with $\operatorname{spt} u \subseteq \Z_{\geq 0}$  there is $M > 0$ with
    \begin{equation*}
        |u_n|_H \leq M \rho^n
        \qquad
        (n \in \Z),
    \end{equation*}
i.e., $u \in \ell_{\infty, \rho}(\Z; H)$.
\end{enumerate}
\end{proposition}

\begin{proof}
$(i) \Rightarrow (ii)$:
From Theorem \ref{t:causal} we deduce that for every $r \in (r(A), 1) \neq \emptyset$ we have $u = (\tau - A)^{-1} \delta_{-1} x \in \ell_{2, \rho}(\Z; H)$.

$(ii) \Rightarrow (iii)$:
This follows as $\ell_{2, \rho}(\Z_{\geq 0}; H) \subseteq \ell_{\infty, \rho}(\Z_{\geq 0}; H)$ for all $\rho > 0$.

$(iii) \Rightarrow (i)$:
Let $\rho \in (0, 1)$, $x \in H$ and $u \in \ell_{\infty, \rho}(\Z; H)$ as in $(iii)$.
Let $\tilde{\rho} > \rho$.
Applying Lemma \ref{l:l1emb}(ii) we obtain $u \in \ell_{1, \tilde{\rho}}(\Z; H)$.   
Also $u_n = A^n x$ for $n \in \Z_{\geq 0}$ by Proposition \ref{p:linear-initial-value-problem}.
We observe that the closed graph theorem implies that the mapping
\begin{equation*}
    H\ni x\mapsto \big(A^n x \big)_{n \in \Z_{\geq 0}} \in \ell_{1,\tilde{\rho}}(\Z_{\geq 0};H)
\end{equation*}
is continuous.
We have
\begin{equation*}
	\sup_{x\in H,|x|_H\leq 1} \sum_{n=0}^\infty |A^n x|_H \tilde{\rho}^{-n} < \infty.
\end{equation*}
For $n\in \Z_{\geq 0}$ we find $x_n\in H$, $|x_n|_H=1$ such that $|A^{n}x_n|_H\geq \frac{1}{2} |A^{n}|_{L(H)}$.
Thus, by the monotone convergence theorem, we deduce
 \begin{align*}
    \sum_{n=0}^\infty |(A/\tilde{\rho})^{n}|_{L(H)} &\leq 2\sum_{n=0}^\infty \sup_{k\in\mathbb{N}} |A^{n}x_k|_H  \tilde{\rho}^{-n}
    \\ & = 2 \sup_{k\in\mathbb{N}} \sum_{n=0}^\infty  |A^{n}x_k|_H \tilde{\rho}^{-n}
    \\ & \leq 2\sup_{x\in H,|x|_H\leq 1} \sum_{n=0}^\infty |A^{n}x|_H \tilde{\rho}^{-n} <\infty. 
 \end{align*}
Using the Neumann series, we obtain $(1-(A/\tilde{\rho}))^{-1}\in L(H)$ and therefore $(\tilde{\rho} - A)^{-1} \in L(H)$.
We can repeat the argument for $\theta A$ in place of $A$ for all $\theta \in S_1$.
Thus, $(z-A)^{-1}\in L(H)$ for all $z\in S_{\tilde{\rho}}$.
As $\rho < 1$, we obtain $r(A) < 1$.
\end{proof}

\section{Stable manifolds}

Stable manifolds of difference equations $u_{n+1} = A u_n + f(u_n)$, $n \in \Z_{\geq 0}$, with hyperbolic linear part $A$ and $f(0) = 0$, can be constructed e.g.\ with the graph transform method or the Lyapunov--Perron method (see e.g.\ \cite{Elaydi2005} and the references therein). In this section we extend the Lyapunov--Perron method to general difference equations $\tau u = Au + F(u)$ (Theorem \ref{t:stable-manifold}).

\begin{theorem}[Existence and uniqueness of solution in $\ell_{2,\rho}(\Z; H)$ with $\rho$ in spectral gap]\label{t:existence2}
Let $A \in L(H)$, $\rho > 0$ with $\sigma(A) \cap S_\rho  = \emptyset$, $F \colon \ell_{2,\rho}(\Z; H) \rightarrow \ell_{2,\rho}(\Z; H)$ with
\begin{equation*}
   |F|_{\operatorname{Lip}(\ell_{2,\rho}(\Z; H))} <1/M_\rho
\end{equation*}
for $M_\rho \coloneqq \sup_{z \in S_\rho} |(z-A)^{-1}|_{L(H)} < \infty$.
Then for each $x \in H$ there exists a unique $u \eqqcolon \mathcal S_\rho(x) \in \ell_{2,\rho}(\Z; H)$ with
\begin{equation*}
   \tau u = Au + F(u) + \delta_{-1} x.
\end{equation*}
\end{theorem}

\begin{proof}
By Theorem \ref{t:causal}, $(\tau - A)^{-1} \in L(\ell_{2,\rho}(\Z; H))$. Using Lemma \ref{lem:mo}, 
\begin{equation*}
   |(\tau - A)^{-1}|_{L(\ell_{2,\rho}(\Z; H))}
   =
   \sup_{z \in S_\rho}
   |(z - A)^{-1}|_{L(H)}
   = M_\rho.   
\end{equation*}
Let $x \in H$. Then
\begin{align*}
   \ell_{2,\rho}(\Z; H) &\to \ell_{2,\rho}(\Z; H)
\\
   u &\mapsto (\tau - A)^{-1}F(u) + (\tau - A)^{-1}\delta_{-1} x
\end{align*}
is a contraction with unique fixed point $\mathcal S_\rho(x) \in \ell_{2,\rho}(\Z; H)$, since for $u, v \in \ell_{2,\rho}(\Z; H)$
\begin{align*}
   & \; \big|
   \big[ (\tau - A)^{-1}F(u) + (\tau - A)^{-1}\delta_{-1} x \big]
   -
   \big[ (\tau - A)^{-1}F(v) + (\tau - A)^{-1}\delta_{-1} x \big]
   \big|_{\ell_{2,\rho}(\Z; H)}
\\[2ex]
   \leq & \; 
   | (\tau - A)^{-1} |_{L(\ell_{2,\rho}(\Z; H))}
   | F(u) - F(v) |_{\ell_{2,\rho}(\Z; H)}
   \leq 
   M_\rho |F|_{\operatorname{Lip}(\ell_{2,\rho}(\Z; H))}
   |u - v|_{\ell_{2,\rho}(\Z; H)}.
\end{align*}
\phantom{ho}
\end{proof}

Let $\cc = \cc(\Z; H) \coloneqq \setm{x \in H^\Z}{x_k \neq 0 \text{ for only finitely many } k \in \Z}$ denote the vector space of sequences in $H^\Z$ with compact support. $\cc \subseteq \ell_{p,\rho}(\Z; H)$ for $1 \leq p \leq \infty$, $\rho > 0$.

\begin{definition}[Causality]\label{d:causal}
A mapping $F \colon \cc(\Z;H) \rightarrow H^{\Z}$ is called \emph{causal}, if for all $u, v \in \cc$, $a \in \Z$,
\begin{equation*}
   \operatorname{spt} (u - v) \subseteq \Z_{\geq a}
   \quad \Rightarrow \quad
   \operatorname{spt} (F(u) - F(v)) \subseteq \Z_{\geq a}.
\end{equation*}
\end{definition}

\begin{corollary}[Solution operator for initial value problems]\label{c:causal_solution}
Let $A \in L(H)$, $\rho > r(A)$, $F \colon \ell_{2,\rho}(\Z; H) \rightarrow \ell_{2,\rho}(\Z; H)$ with $F(0)=0$, $F|_{\cc}$ causal, and
\begin{equation*}
   |F|_{\operatorname{Lip}(\ell_{2,\rho}(\Z; H))} <1/M_\rho
\end{equation*}
for $M_\rho \coloneqq \sup_{z \in S_\rho} |(z-A)^{-1}|_{L(H)} < \infty$.
Then for each $x \in H$ there exists a unique $u \eqqcolon \mathcal S_\rho(x) \in \ell_{2,\rho}(\Z; H)$ with
\begin{equation*}
   \tau u = Au + F(u) + \delta_{-1} x
\end{equation*}
and the \emph{solution operator} $\mathcal S_\rho$ satisfies
\begin{equation*}
   \operatorname{spt} \mathcal S_\rho(x) \subseteq \Z_{\geq 0}.
\end{equation*}
\end{corollary}

\begin{proof} \emph{Step 1:} Show that $F \colon \ell_{2,\rho}(\Z; H) \rightarrow \ell_{2,\rho}(\Z; H)$ is causal, i.e.
\begin{equation*}
   \forall u, v \in \ell_{2,\rho}(\Z; H), a \in \Z \colon
   \operatorname{spt} (u-v) \subseteq \Z_{\geq a}
   \; \Rightarrow \;
   \operatorname{spt} (F(u) - F(v)) \subseteq \Z_{\geq a}.
\end{equation*}
To this end, let $u, v \in \ell_{2,\rho}(\Z; H)$ and $a \in \Z$ such that $\operatorname{spt} (u-v) \subseteq \Z_{\geq a}$. We define sequences $(u^{(k)})_{k \in \N}, (v^{(k)})_{k \in \N}$ in $\cc$ by
\begin{equation*}
   u_j^{(k)}
   \coloneqq
   \begin{cases}
      u_j & \text{if } j \in [-k,k],
   \\
      0 & \text{if } j \in \Z \setminus [-k,k],
   \end{cases}
   \qquad (k \in \N),
\end{equation*}
and similarly for $v_j^{(k)}$. Then $u^{(k)} \to u$, $v^{(k)} \to v$ in $\ell_{2,\rho}(\Z; H)$ and for all $k \in \N$ we have $\operatorname{spt} (u^{(k)} - v^{(k)}) \subseteq \Z_{\geq a}$. Using the fact that $F|_{\cc}$ is causal,
\begin{equation*}
      \operatorname{spt} (F(u^{(k)}) - F(v^{(k)}) )
      \subseteq
      \Z_{\geq a}
      \qquad
      (k \in \N).
\end{equation*}
$F$ is Lipschitz continuous and hence the limit $\lim_{k \to \infty} (F(u^{(k)}) - F(v^{(k)})) = F(u) - F(v)$ also satisfies $\operatorname{spt} (F(u) - F(v)) \subseteq \Z_{\geq a}$.

\emph{Step 2:}
Let $x \in H$.
As in the proof of Theorem \ref{t:existence2} we see that $\mathcal S_\rho(x)$ is the unique fixed point of the contraction
\begin{align*}
   \ell_{2,\rho}(\Z; H) &\to \ell_{2,\rho}(\Z; H) ,
\\
   u &\mapsto (\tau - A)^{-1}F(u) + (\tau - A)^{-1}\delta_{-1} x .
\end{align*}
To show that $\operatorname{spt} \mathcal S_\rho(x) \subseteq \Z_{\geq 0}$, let $u \in \ell_{2, \rho}(\Z_{\geq 0}; H)$ and show that the fixed point iteration preserves the support of $u$, i.e.
\begin{equation}\label{e:pres}
    \operatorname{spt} \big((\tau - A)^{-1}F(u) + (\tau - A)^{-1}\delta_{-1} x\big) \subseteq \Z_{\geq 0}.
\end{equation}
By Theorem \ref{t:causal}$(ii)$ we know that $\operatorname{spt} (\tau - A)^{-1}\delta_{-1} x \subseteq \Z_{\geq 0}$.
We have $\operatorname{spt} (u - 0) \subseteq \Z_{\geq 0}$ and as $F(0) = 0$ we have seen in Step 1 that $\operatorname{spt} F(u) = \operatorname{spt} (F(u) - F(0)) \subseteq \Z_{\geq 0}$.
Since $(\tau - A)^{-1}$ is causal by Theorem \ref{t:causal}$(i)$, we deduce that $\operatorname{spt} (\tau - A)^{-1} F(u) \subseteq \Z_{\geq 0}$, proving \eqref{e:pres}.
\end{proof}

\begin{remark}[Riesz projection {\cite[Proposition 6.9]{Hislop:Sigal1996}}]
Let $A \in L(H)$, $\gamma \in \R_{>0}$. If $\sigma(A) \cap S_\gamma  = \emptyset$, then the Riesz projections
\begin{equation*}
   P_\gamma^+
   \coloneqq
   \frac{1}{2 \pi i}
   \int_{S_\gamma} (z - A)^{-1} \,\mathrm{d}z
   \in L(H)
   \qquad \text{and} \qquad
   P_\gamma^-
   \coloneqq
   I - P_\gamma^+ \in L(H)
\end{equation*}
satisfy
\begin{enumerate}
   \item[(i)] $(P_\gamma^\pm)^2 = P_\gamma^\pm$, $P_\gamma^+[H] \oplus P_\gamma^-[H] = H$.
   \item[(ii)] $P_\gamma^\pm A = A P_\gamma^\pm$, 
      $A P_\gamma^\pm[H] = P_\gamma^\pm[H]$, i.e.\ the subspaces $P_\gamma^\pm[H]$ are invariant under $A$.
   \item[(iii)] $\sigma(AP_\gamma^+) = \sigma(A) \cap B(0, \gamma)$, $\sigma(AP_\gamma^-) = \sigma(A) \setminus B(0, \gamma)$.
\end{enumerate}
Moreover, for $\gamma_1, \gamma_2 \in \R_{>0}$ with $\gamma_1 < \gamma_2$,
\begin{enumerate}
   \item[(iv)] $\forall \gamma \in [\gamma_1,\gamma_2] \colon \sigma(A) \cap S_\gamma  = \emptyset
   \;\Rightarrow\;
   P_{\gamma_1}^\pm = P_{\gamma_2}^\pm$.
\end{enumerate}
\end{remark}

\begin{theorem}[Lyapunov--Perron operator]\label{t:Lyapunov--Perron}
Under the assumptions of Theorem \ref{t:existence2}, the \emph{Lyapunov--Perron operator}
\begin{align*}
   \mathcal L_\rho \colon H \times \ell_{2,\rho}(\Z; H) &\to \ell_{2,\rho}(\Z; H)
\\
   (x, u) &\mapsto \chi_{\Z_{\geq 0}} (\tau - A)^{-1} (F(u) + \delta_{-1}x)
\end{align*}
is well-defined and for $x \in H$, $\mathcal L_\rho(x, \cdot)$ is a contraction.
\end{theorem}

\begin{proof}
The proof is along the lines of Theorem \ref{t:existence2}.
\end{proof}

\begin{definition}[Hyperbolic]
    Let $A \in L(H)$.
    The operator $A$ is called \emph{hyperbolic}, if $\sigma(A) \cap S_1 = \emptyset$.
\end{definition}

\begin{remark}\label{r:spectrum-sum-convergence}
    Let $A\in L(H)$ such that $\sigma(A)\subseteq \C\setminus B[0,1]$.
    Then by definition of the spectrum, $A^{-1}\in L(H)$ exists. Also $\sigma(A^{-1})\subseteq B(0,1)$ and as the spectrum is closed, $r(A^{-1}) < 1$.
    Let $x\in H$ such that $(A^nx)_{n\in\Z_{\geq 0}}\in \ell_2(\Z_{\geq 0};H)$.
    It follows that $x = 0$.
    Indeed there is an $N\in\Z_{\geq 0}$ such that for every $n>N$
    \begin{equation*}
        |x|_H = |A^{-n}A^nx|_H \leq |A^{-n}|_H|A^nx|_H \leq |A^nx|_H
    \end{equation*}
    by the definition of the spectral radius.
    However, for $(A^nx)_{n\in\Z_{\geq 0}}$ to be square summable, we necessarily have $\lim_{n\rightarrow\infty}|A^nx|_H = 0$.
\end{remark}

\begin{lemma} \label{l:compute-fixed-point-operator}
Let $A \in L(H)$ be hyperbolic.
We write $P \coloneqq P_1^+$, $Q \coloneqq P_1^-$, for the Riesz-projections.
Then for every $v \in \ell_2(\Z; H)$ we have for $n \in \Z$
\begin{align*}
	&(\tau - PAP)^{-1}v = \left(\sum_{k = -\infty}^{n-1} (PAP)^{n-1-k}v_k\right)_n,
	\\
    &(\tau - QAQ)^{-1}v = \left(-\sum_{k = n}^\infty (QAQ)^{n-1-k}v_k\right)_n.
\end{align*}
\end{lemma}

\begin{proof}
Note that $PAP = AP = A|_{P[H]}$ on $P[H]$ and $\sigma(A|_{P[H]}) = \sigma(A) \cap B(0, 1)$ by the Riesz-projection theorem.
Therefore $r(A|_{P[H]}) < 1$.
Similarly $QAQ = AQ = A|_{Q[H]}$ on $Q[H]$ and $\sigma(A|_{Q[H]}) = \sigma(A) \setminus B[0, 1]$.
In particular, $0$ is in the resolvent set of $A|_{Q[H]}$ and $r((A|_{Q[H]})^{-1}) < 1$.
By definition of the spectral radius both sums of Lemma \ref{l:compute-fixed-point-operator} exist for all $v \in \ell_2(\Z; H)$.

For every $v \in \ell_2(\Z;H)$ we compute for $n \in \Z$
\begin{align*}
    (\tau - PAP) &\left(\sum_{k = -\infty}^{n-1} (PAP)^{n-1-k}v_k\right)_n
    \\
    = &\left(\sum_{k = -\infty}^n (PAP)^{n-k}v_k - \sum_{k = -\infty}^{n-1} (PAP)^{n-k}v_k\right)_n = v
\end{align*}
and
\begin{align*}
    (\tau - QAQ) &\left(\sum_{k = n}^{\infty} (QAQ)^{n-1-k}v_k\right)_n
    \\
    = &\left(\sum_{k = n+1}^{\infty} (QAQ)^{n-k}v_k - \sum_{k = n}^{\infty} (QAQ)^{n-k}v_k\right)_n = -v.
\end{align*}
The assertion follows by rearranging the terms.
\end{proof}

\begin{theorem}[Characterization of Lyapunov--Perron fixed point]\label{t:Lyapunov--Perron-FP}
Let $A \in L(H)$ be hyperbolic, $F \colon \ell_2(\Z; H) \rightarrow \ell_2(\Z; H)$ with
\begin{equation*}
	|F|_{\operatorname{Lip}(\ell_2(\Z;H))} < 1/M_1
\end{equation*}
for $M_1 \coloneqq \sup_{z \in S_1} |(z-A)^{-1}|_{L(H)} < \infty$.
Let $\xi \in P[H]$ and  $u \in  \ell_{2}(\Z; H)$. 
Then the following statements are equivalent:

(i) $u$ is the unique fixed point of the Lyapunov--Perron operator $\mathcal L_1(\xi, \cdot)$.

(ii) $\operatorname{spt} u \subseteq \Z_{\geq 0}$ and for $n \in \Z_{\geq 0}$
\begin{align*}
   P u_n &= (PAP)^n \xi + \sum_{k = -\infty}^{n-1} (PAP)^{n-1-k}PF(u)_k,
\\
   Q u_n &= - \sum_{k = n}^{\infty} (QAQ)^{n-1-k} QF(u)_k.
\end{align*}
\end{theorem}

\begin{proof}

$(i)\Rightarrow(ii)$:
Let $u \in \ell_2(\Z; H)$ be the unique fixed point.
By definition of the Lyapunov--Perron operator, it follows that $\operatorname{spt}(u)\subseteq\Z_{\geq 0}$.
Let $n\geq 0$.
We have $((\tau - A)^{-1}\delta_{-1}\xi)_n = A^n\xi$ and $P\xi = \xi$ and $Q\xi = 0$.
Using the equivalence $(iii) \Leftrightarrow (v)$ of Proposition \ref{p:linear-initial-value-problem}
\begin{align*}
    &Pu_n = (P\chi_{\Z_{\geq 0}}(\tau - A)^{-1}(F(u) + \delta_{-1}\xi))_n = (\chi_{\Z_{\geq 0}}(\tau - A)^{-1}PF(u))_n + (PAP)^n\xi,
    \\
    &Qu_n = (Q\chi_{\Z_{\geq 0}}(\tau - A)^{-1}(F(u) + \delta_{-1}\xi))_n = (\chi_{\Z_{\geq 0}}(\tau - A)^{-1}QF(u))_n,
\end{align*}
as $P$ and $Q$ commute with all linear operators involved.
Thus we have to compute $(\tau - A)^{-1}PF(u)$ and $(\tau - A)^{-1}QF(u)$ to show the first and second equation of $(ii)$.

Using the properties of the Riesz-projections, we compute
\begin{align*}
	&P(\tau - A)^{-1} = (\tau - A)^{-1}P = (\tau - PAP)^{-1}P,
	\\
    &Q(\tau - A)^{-1} = (\tau - A)^{-1}Q = (\tau - QAQ)^{-1}Q.
\end{align*}
Therefore $(ii)$ follows when setting either $v = PF(u)$ or $v = QF(u)$ in Lemma \ref{l:compute-fixed-point-operator}.

$(ii)\Rightarrow(i)$:
For every $n\geq 0$ we compute
\begin{align*}
    u_n &= Pu_n + Qu_n
    \\
        &= (PAP)^n\xi + \sum_{k = -\infty}^{n-1} (PAP)^{n-1-k}PF(u)_k - \sum_{k = n}^{\infty}(QAQ)^{n-1-k}QF(u)_k
    \\
        &= A^n\xi + \sum_{k = 0}^{n-1} (PAP)^{n-1-k}PF(u)_k + \sum_{k=0}^{n-1}(QAQ)^{n-1-k}QF(u)_k
    \\
            &\quad + \sum_{k = -\infty}^{-1} (PAP)^{n-1-k}PF(u)_k - \sum_{k = 0}^{\infty}(QAQ)^{n-1-k}QF(u)_k
    \\
        &= A^n\xi + \sum_{k = 0}^{n-1} (PA^{n-1-k}F(u)_k + QA^{n-1-k}F(u)_k)
    \\
			&\quad + A^n\left(\sum_{k = -\infty}^{-1} (PAP)^{-1-k}PF(u)_k - \sum_{k = 0}^{\infty}(QAQ)^{-1-k}QF(u)_k\right)
    \\
        &= A^n\left(\xi + \sum_{k = -\infty}^{-1} (PAP)^{-1-k}PF(u)_k - \sum_{k = 0}^{\infty}(QAQ)^{-1-k}QF(u)_k\right)
    \\
        &\qquad + \sum_{k = 0}^{n-1} A^{n-1-k}F(u)_k.
\end{align*}
By variation of constants we see that $u_{n+1} = Au_n + F(u)_n = Au_n + F(u)_n + \delta_{-1,n}\xi$ for every $n \geq 0$.
That is $u_n = ((\tau - A)^{-1}(F(u) + \delta_{-1}\xi))_n$ for all $n>0$.
Furthermore we have 
\begin{align*}
	\big((\tau - A)^{-1}(F(u) + \delta_{-1}\xi)\big)_0 &= \big((\tau - A)^{-1}(P+Q)(F(u) + \delta_{-1}\xi)\big)_0
	\\
		&= \big((\tau - A)^{-1}PF(u)\big)_0 + \xi + \big((\tau - A)^{-1}QF(u)\big)_0.
\end{align*}
We have computed the expressions $((\tau - A)^{-1}PF(u))_0$ and $((\tau - A)^{-1}QF(u))_0$ in Lemma \ref{l:compute-fixed-point-operator} and they coincide with those of $u_0$.
Therefore $u = \chi_{\Z_{\geq 0}} (\tau - A)^{-1}(F(u) + \delta_{-1}\xi)$, since $\operatorname{spt}(u)\subseteq\Z_{\geq 0}$.
\end{proof}

\begin{definition}[Admissibility condition]
Let $\rho_1, \rho_2 > 0$.
\\
Then $F \colon \cc(\Z; H) \to \ell_{2,\rho_1}(\Z; H) \cap \ell_{2,\rho_2}(\Z; H)$ is called \emph{$(\rho_1, \rho_2)$-admissible} if
\begin{equation*}
   F \colon  \cc(\Z; H) \subseteq \ell_{2,\rho_i}(\Z; H) \to \ell_{2,\rho_i}(\Z; H)
\end{equation*}
is Lipschitz for $i \in \set{1,2}$.

We denote the unique Lipschitz extensions of $F$ to $\ell_{2,\rho_i}(\Z; H)$ for $i \in \set{1,2}$ again by $F$.
\end{definition}

We are now in a position to prove the existence of a stable manifold of the trivial solution $u = 0$.

\begin{theorem}[Stable manifold]\label{t:stable-manifold}
Let $A \in L(H)$ be hyperbolic, $\rho > r(A)$.
Let $F$ be $(1, \rho)$-admissible, such that
\begin{equation*}
   |F|_{\operatorname{Lip}(\ell_{2}(\Z; H))} <1/M_1
   \quad \text{and} \quad
   |F|_{\operatorname{Lip}(\ell_{2,\rho}(\Z; H))} <1/M_\rho.
\end{equation*}
We also assume that $F$ is causal and $F(0) = 0$.
Let $P \coloneqq P_1^+$, $Q \coloneqq P_1^-$ denote the Riesz projections, $T(\xi) \in \ell_{2}(\Z; H)$ the unique fixed point of the Lyapunov--Perron operator $\mathcal L_1(\xi, \cdot)$ for $\xi \in P[H]$.
Then the stable set
\begin{equation*}
   W^s \coloneqq \setm{x \in H}{\mathcal S_\rho(x) \in \ell_{2}(\Z; H)}
\end{equation*}
is isomorphic to the graph of 
$w^s \colon P[H] \rightarrow Q[H]$, $\xi \mapsto Q \big(T(\xi)(0)\big)$,
\begin{equation*}
   W^s = \setm{\xi + w^s(\xi) \in H}{\xi \in P[H]},
\end{equation*}
i.e.\ $W^s$ is a manifold, called \emph{stable manifold (of $\tau u = Au + F(u) + \delta_{-1} x$ in $\ell_{2,\rho}(\Z; H)$ at the fixed point $0$)}.
\end{theorem}

\begin{proof}
An important observation is that as $F$ is causal and $F(0) = 0$, we deduce that for all $u \in \ell_2(\Z;H)$
    with $\operatorname{spt} u = \operatorname{spt}(u -  0) \subseteq \Z_{\geq 0}$ we have $\operatorname{spt} F(u) = \operatorname{spt}(F(u) - F(0)) \subseteq\Z_{\geq 0}$.
This simplifies the first sum in Theorem \ref{t:Lyapunov--Perron-FP}$(ii)$.

$(\subseteq)$:
Let $x\in W^s$.
We set $\xi\coloneqq P(x)\in P[H]$ and $\eta\coloneqq Q(x)\in Q[H]$, such that $x = \xi + \eta$.
We will prove that $\eta = w^s(\xi) = Q(T(\xi)(0))$ and even more that $u \coloneqq \mathcal S_\rho(x) = T(\xi)$ (note that $(\mathcal S_\rho(x))_0 = x = \xi + \eta$).
By Corollary \ref{c:causal_solution} we have $\operatorname{spt}(u) \subseteq \Z_{\geq 0}$, and as $x \in W^s$ we know that $u \in \ell_2(\Z;H)$.

We prove the two equalities of Theorem \ref{t:Lyapunov--Perron-FP}$(ii)$.
Let $n \geq 0$.
By Theorem \ref{t:initial_value_problem} we have the variation of constants formula
\begin{equation}\label{vofc}
    u_n = A^nx + \sum_{k = 0}^{n-1} A^{n-1-k}F(u)_k.
\end{equation}
Multiplying \eqref{vofc} with $P$ results in
\begin{align*}
    Pu_n &= (PAP)^n\xi + \sum_{k=0}^{n-1} (PAP)^{n-1-k}PF(u)_k
    \\
    &= (PAP)^n\xi + \sum_{k=-\infty}^{n-1} (PAP)^{n-1-k}PF(u)_k.
\end{align*}
By multiplying \eqref{vofc} with $Q$, we compute
\begin{align*}
    Qu_n &= (QAQ)^n\eta + \sum_{k=0}^{n-1} (QAQ)^{n-1-k}QF(u)_k
    \\
    & = (QAQ)^n\eta + \left(\sum_{k=0}^{\infty}(QAQ)^{n-1-k}QF(u)_k - \sum_{k=n}^{\infty}(QAQ)^{n-1-k}QF(u)_k\right)
    \\
    & = (QAQ)^n\left(\eta + \sum_{k=0}^{\infty}(QAQ)^{-1-k}QF(u)_k\right) - \sum_{k=n}^{\infty}(QAQ)^{n-1-k}QF(u)_k).
\end{align*}
By Lemma \ref{l:compute-fixed-point-operator} we know that $\left( \sum_{k=n}^{\infty}(QAQ)^{n-1-k}QF(u)_k) \right)_n = (\tau - QAQ)^{-1}QF(u) \in \ell_2(\Z;H)$ and so
\begin{equation*}
    \left((QAQ)^n\left(\eta + \sum_{k=0}^{\infty}(QAQ)^{-1-k}QF(u)_k\right)\right)_n\in \ell_2(\Z;H).
\end{equation*}
In Remark \ref{r:spectrum-sum-convergence} we have noted that necessarily $\eta + \sum_{k=0}^{\infty}(QAQ)^{-1-k}QF(u)_k) = 0$.
By Theorem \ref{t:Lyapunov--Perron-FP}, $u$ is the unique fixed point of the Lyapunov--Perron operator $\mathcal L_1(\xi, \cdot)$.

$(\supseteq)$:
Let $(\xi,\eta)\in \operatorname{graph}(w^s)$, that is $\xi\in P[H]$ and $Q(T(\xi)_0) = - \sum_{k = 0}^{\infty} (QAQ)^{-1-k} \cdot QF(u)_k = \eta$ by Theorem \ref{t:Lyapunov--Perron-FP}$(ii)$. 
We will verify that $\xi+\eta\in W^s$, that is $\mathcal S_\rho(\xi+\eta)\in \ell_2(\Z;H)$.
To this end we show that $u\coloneqq T(\xi) = \mathcal S_\rho(\xi+\eta)$.
By the variation of constants formula  \eqref{vofc} we know that
\begin{align*}
    u = \mathcal S_\rho(\xi + \eta)  \quad\Leftrightarrow\quad  &u\in \ell_{2,\rho}(\Z;H) \textrm{ and } \operatorname{spt}(u)\subseteq\Z_{\geq 0}\\
        &\textrm{and } u_n = A^n(\xi + \eta) + \sum_{k = 0}^{n-1} A^{n-1-k} F(u)_k.
\end{align*}
By Theorem \ref{t:Lyapunov--Perron-FP}$(ii)$ we have $\operatorname{spt}(u)\subseteq\Z_{\geq 0}$.
Therefore if $\rho > 1$ we see that $u\in \ell_2(\Z_{\geq 0}; H)\subseteq \ell_{2,\rho}(\Z_{\geq 0}; H)$.
If $\rho < 1$ the statement of the theorem is trivial as $W^s = H$ on the one hand and $P = I$ and $Q = 0$ on the other hand. It remains to establish the formula for $u_n$. For this, we see that Theorem \ref{t:Lyapunov--Perron-FP}$(ii)$ also yields
\begin{align*}
    u_n &= Pu_n + Qu_n
    \\
    &= (PAP)^n\xi + \sum_{k = -\infty}^{n-1} (PAP)^{n-1-k}PF(u)_k - \sum_{k=n}^\infty (QAQ)^{n-1-k}QF(u)_k
    \\
    &= A^n\xi + P\sum_{k=0}^{n-1}A^{n-1-k}F(u)_k + Q\sum_{k=0}^{n-1}A^{n-1-k}F(u)_k - \sum_{k = 0}^\infty (QAQ)^{n-1-k}QF(u)_k
    \\
    &= A^n\xi + \sum_{k=0}^{n-1}A^{n-1-k}F(u)_k - A^n\sum_{k=0}^\infty(QAQ)^{-1-k}QF(u)_k
    \\
    &= A^n \left(\xi - \sum_{k=0}^\infty(QAQ)^{-1-k}QF(u)_k\right) + \sum_{k=0}^{n-1}A^{n-1-k}F(u)_k
    \\
    &= A^n(\xi+\eta) + \sum_{k=0}^{n-1}A^{n-1-k}F(u)_k.
\end{align*}
\phantom{jp}
\end{proof}

\bibliography{difference.bib}{}

\begin{thebibliography}{1}

\bibitem{Elaydi2005}
S.~Elaydi.
\newblock {\em An Introduction to Difference Equations}.
\newblock Undergraduate Texts in Mathematics. Springer New York, third edition,
  2005.

\bibitem{Hislop:Sigal1996}
P.~{Hislop} and I.~{Sigal}.
\newblock {\em {Introduction to spectral theory. With applications to
  Schr\"odinger operators.}}
\newblock New York, NY: Springer-Verlag, 1996.

\bibitem{Kalauch:Picard:Siegmund:Trostorff:Waurick2014}
A.~Kalauch, R.~Picard, S.~Siegmund, S.~Trostorff, and M.~Waurick.
\newblock A {H}ilbert space perspective on ordinary differential equations with
  memory term.
\newblock {\em Journal of Dynamics and Differential Equations}, 26(2):369--399,
  2014.

\bibitem{Katznelson2004}
Y.~{Katznelson}.
\newblock {\em {An introduction to harmonic analysis.}}
\newblock Cambridge: Cambridge University Press, 3rd edition, 2004.

\bibitem{Picard:Trostorff:Waurick2014}
R.~Picard, S.~Trostorff, and M.~Waurick.
\newblock A functional analytic perspective to delay differential equations.
\newblock {\em Operators and Matrices}, 8(1):217--236, 2014.

\bibitem{Picard:McGhee2011}
R.~H. Picard and D.~F. McGhee.
\newblock {\em Partial Differential Equations: A Unified Hilbert Space
  Approach}.
\newblock De Gruyter expositions in mathematics. De Gruyter, 2011.

\end{thebibliography}
\bibliographystyle{abbrv}

\end{document}